\documentclass[10pt,reqno]{amsart}
\usepackage[utf8]{inputenc}
\usepackage{amsmath,times,hyperref, multicol}
\usepackage{amsthm}
\usepackage{amssymb}
\usepackage{amsfonts}
\usepackage{latexsym}
\usepackage{graphicx}
\usepackage[usenames,dvipsnames]{color} 
\usepackage{enumitem}
\usepackage{halloweenmath}
\setlist[itemize]{leftmargin=*}
\setlist[enumerate]{leftmargin=*}
\usepackage[font={small}]{caption}
\usepackage{subcaption}
\usepackage{cite}
\usepackage{comment}
\usepackage{mathtools}
\mathtoolsset{showonlyrefs}
\usepackage{amsaddr}

\makeatletter
\newenvironment{leftalign*}[1][\parindent]{\setlength\hangindent{#1}\start@align\tw@\st@rredtrue\m@ne}{\endalign}

\makeatother

\usepackage{tikz,amsmath,amsfonts}
\usetikzlibrary{calc,intersections}
\usetikzlibrary{decorations.pathreplacing}
\usepackage{ifthen}
\renewcommand{\phi}{\varphi}

\newcommand{\T}{\mathbb{T}}

\renewcommand{\hat}{\widehat}

\let\oldenumerate=\enumerate
	\def\enumerate{
	\oldenumerate
	\setlength{\itemsep}{5pt}
	}
\let\olditemize=\itemize
	\def\itemize{
	\olditemize
	\setlength{\itemsep}{5pt}
	}

\newtheorem{Theorem}{Theorem}[section]
\newtheorem{Lemma}[Theorem]{Lemma}
\newtheorem{Proposition}[Theorem]{Proposition}
\newtheorem{Corollary}[Theorem]{Corollary}

\newtheorem{Question}[Theorem]{Question}
\newtheorem{Example}[Theorem]{Example}
\newtheorem{Remark}[Theorem]{Remark}
\newtheorem{Conjecture}[Theorem]{Conjecture}

\allowdisplaybreaks

\numberwithin{equation}{section}

\begin{document}

\title{Crouzeix's conjecture for classes of matrices}

\author{Ryan O'Loughlin}
\address{School of Mathematics\\
University of Leeds\\
Leeds, U.K. LS2 9JT}
\email[]{ryan.oloughlin.1@ulaval.ca}

\author{Jani Virtanen}
\address{Department of Mathematics and Statistics\\
University of Reading\\
Reading, U.K.\\}
\email[]{j.a.virtanen@reading.ac.uk}
\date{}

\dedicatory{Dedicated to Albrecht B\"ottcher on the occasion of his seventieth birthday.}

\begin{abstract}
For a matrix $A$ which satisfies Crouzeix's conjecture, we construct several classes of matrices from $A$ for which the conjecture will also hold. We discover a new link between cyclicity and Crouzeix's conjecture, which shows that Crouzeix's Conjecture holds in full generality if and only if it holds for the differentiation operator on a class of analytic functions. We pose several open questions, which if proved, will prove Crouzeix's conjecture. We also begin an investigation into Crouzeix's conjecture for symmetric matrices and in the case of $3 \times 3$ matrices, we show Crouzeix's conjecture holds for symmetric matrices if and only if it holds for analytic truncated Toeplitz operators.
\end{abstract}

\maketitle

\section{Introduction}
Finding an upper bound for the norm of an operator is one of the most fundamental endeavours in functional analysis and Crouzeix's conjecture provides a computational geometric approach to this. Despite its simplicity and strong numerical evidence, the conjecture has not been proved. The purpose of this article is to survey recent progress on the conjecture, as well as provide several classes of matrices which satisfy Crouzeix's conjecture and introduce novel approaches to study the conjecture. In particular, we discover a new link between cyclicity and Crouzeix's conjecture. This affirms the conjecture in the positive for several new classes of matrices and gives elegant shortened algebraic proofs of several previously known results obtained by detailed complex analysis techniques. Our investigation leads to several open questions, which if proved, will prove Crouzeix's conjecture.

Define the numerical range of an $n \times n$ matrix $A$ with complex entries by 
$$
W(A) := \{ \langle Ax , x \rangle : \langle x , x \rangle   = 1 \},
$$ 
where $\langle \cdot , \cdot \rangle$ refers to the Euclidean inner product on $\mathbb{C}^n$. Crouzeix's conjecture can be stated as follows.
\begin{Conjecture}\label{CC}
For all square complex matrices $A$ and all complex polynomials $p$, 
\begin{equation}\label{ECC}
    \|p(A)\| \leq 2 \sup _{z \in W(A)}|p(z)|,
\end{equation}
where $ \| \cdot \|$ denotes the standard operator norm for matrices.
\end{Conjecture}
\noindent

Crouzeix \cite{CC2matrix, CC11.08} showed that for each polynomial $p$, $ \|p(A)\| \leq 11.08 \sup _{z \in W(A)}|p(z)|$ and the bound was later improved to $1 + \sqrt{2}$ by Crouzeix and Palencia \cite{1root2}. Using the arguments developed by Crouzeix and Palencia, in \cite{CCpalenciaextension} Caldwell, Greenbaum and Li improved the bound to $2$ in some special cases. Since the operator norm and numerical range of a square matrix $A$ are invariant under unitary equivalence Conjecture \ref{CC} holds for $A$ if and only if it holds for all matrices in the unitary equivalence class of $A$. This is a long exploited fact that we will use throughout. Consequently it can be shown that all normal matrices satisfy the conjecture. Further, the conjecture has been shown to hold for several other classes of matrices including $2 \times 2$ matrices \cite{CC2matrix} (by Crouzeix), certain tridiagonal matrices \cite{CC3by3} (by Glader, Kurula and Lindstr\"{o}m), certain contractions with eigenvalues that are sufficiently well-separated \cite{CCgorkin} (by Bickel, Gorkin, Greenbaum, Ransford, Schwenninger and Wegert) and numerous other specialised cases. Numerical investigations have also strongly supported the conjecture (see, for example, the work of Greenbaum and Overton \cite{CCnum} and the references therein). Many results on the conjecture come from works in the more general setting of $K$-spectral sets, and other related questions involving intersections of $K$-spectral sets have also generated research interest (see for example the work of Badea, Beckermann, and Crouzeix \cite{spectral1}).

Throughout, all matrices are assumed to be square. We denote by $\text{conv} \{\mathcal{X}\}$ the convex hull of the set $\mathcal{X}$, by $W(A)^{o}$ the interior of $W(A)$, by $I_d$ the identity matrix, and by $\mathbb{D}$ the open unit disk in $\mathbb{C}$. 

This paper is organised as follows. Section \ref{2} first gives a background on some techniques which have led to partial proofs of Conjecture \ref{CC}. It then provides the preliminary results that we need for our study, and surveys recent relevant new results on the conjecture. In Section \ref{3}, 
we construct classes of matrices for which the conjecture holds. We then develop norm inequalities which prove several results relating cyclicity and Conjecture \ref{CC}, and as a consequence show that in order to prove Conjecture \ref{CC} it suffices to prove the conjecture for a differentiation operator on a specified space of entire functions. Section \ref{4} discusses the conjecture for symmetric matrices (i.e. matrices which are self transpose). In particular, we show that the conjecture holds for $3 \times 3$ symmetric matrices if and only if it holds for $3 \times 3$ truncated Toeplitz operators, and finally we also show that truncated Toeplitz operators serve as model operators for numerical ranges of $3 \times 3$ matrices.

\section{Previous work on Crouzeix's conjecture and some extensions}\label{2}

A commonly used approach to show certain matrices satisfy Conjecture \ref{CC} is to exploit von Neumann's inequality. For a contractive matrix $A$, von Neumann's inequality states that for an analytic function $g: \mathbb{D} \rightarrow \mathbb{C}$ which is continuous up to the boundary $\partial \mathbb{D}$ 
$$
\begin{gathered}
\|g(A)\| \leq  \sup _{z \in \mathbb{D}}|g(z)|. \\
\end{gathered}
$$
Let $\phi: W(A)^{o} \to \mathbb{D}$ be a bijective conformal mapping extended to a homeomorphism of $W(A)$ onto $\overline{\mathbb{D}}$ and let $X$ be an invertible matrix of the same dimension as $A$ such that $X \phi(A) X^{-1}$ is a contraction and
$$
\kappa(X)=\|X\| \cdot\left\|X^{-1}\right\| \leq 2 .
$$
Then, by von Neumann's inequality, for any polynomial $p$,
$$
\begin{aligned}
\|p(A)\| & =\left\|X^{-1}\left(p \circ \phi^{-1}\left(X \phi(A) X^{-1}\right)\right) X\right\| \\
& \leq 2 \max _{z \in \overline{\mathbb{D}}}\left|p \circ \phi^{-1}(z)\right|=2 \max _{z \in W(A)}|p(z)|.
\end{aligned}
$$
Thus, $A$ satisfies Conjecture \ref{CC}. The difficulty of the above approach is finding $\phi$ and $X$ with the requirements specified above. Nonetheless, this approach was used in \cite{CC3by3} to prove Conjecture \ref{CC} for certain $3 \times 3$ matrices.

A similar yet alternative approach to proving Conjecture \ref{CC} for a diagonalisable matrix $A$ is to  write $A = X \Lambda X^{-1}$ where $\Lambda $ is a diagonal matrix with the eigenvalues of $A$ on the diagonal and $X$ is a invertible matrix such that $\| X \| \| X^{-1} \| \leqslant 2$. Then for each polynomial $p$, we have $\| p(A) \| =\| X p(\Lambda) X^{-1} \|$ and
\begin{equation}\label{first}
    \| X p(\Lambda) X^{-1} \| \leqslant \| X \| \| X^{-1} \| \|  p(\Lambda) \| \leqslant 2 \|  p(\Lambda) \| \leqslant 2 \sup_{z \in \sigma(A)} |p(z)| \leqslant 2 \sup_{z \in W(A)} |p(z)|.
\end{equation}
Thus Conjecture \ref{CC} holds for $A$. This approach was used in \cite{CCgorkin}.

We wish to highlight that this approach can be generalised to the case where $A$ is similar to block diagonal matrices. Assume $A = X B X^{-1}$ where $B  = \operatorname{diag} ( B_1, B_2, \ldots B_n)$ and each $B_i$ is a square matrix such that for all polynomials $p$, 
$$\| p (B_i) \| \leqslant b \sup_{z \in W(A)} |p (z) |$$ for some $b \leqslant 2$ and where $\| X^{-1} \| \| X \| b \leqslant 2$. Then $\| p(A) \| = \| X p(B) X^{-1} \|$ and
$$
\| X p(B) X^{-1} \| \leqslant \| X \| \| X^{-1} \| \|  p(B) \| = \| X \| \| X^{-1} \| \max_i \|  p(B_i) \| \leqslant 2 \sup_{z \in W(A)} |p(z)|.
$$
Thus, under these assumptions Conjecture \ref{CC} will hold for $A$. 

For a matrix $A$, denote by $\mathbf{A}(W(A))$ the algebra of functions which are analytic on the interior $W(A)^o$ of $W(A)$ and continuous on the boundary of $W(A)$. There are at least four slightly different equivalent formulations of Conjecture \ref{CC} appearing in the literature. For the sake of completeness we give a proof of the following proposition.

\begin{Proposition}\label{allequiv}
     The following are equivalent: 
    \begin{enumerate}
        \item For all complex polynomials $p$, 
\begin{equation}\label{ECC2time}
    \|p(A)\| \leq 2 \sup _{z \in W(A)}|p(z)|.
\end{equation}
    \item For all functions $f$, analytic on an open neighbourhood of $W(A)$
    \begin{equation}\label{ECC3time}
    \|f(A)\| \leq 2 \sup _{z \in W(A)}|f(z)|.
\end{equation}
    \item For all functions $f$, analytic on an open neighbourhood of $W(A)$ such that \\$\sup _{z \in W(A) }|f(z)| =1 $,
    \begin{equation}
    \|f(A)\| \leq 2 .
    \end{equation}
    \item For all functions $f \in \mathbf{A}(W(A))$,
    \begin{equation}\label{ECC5time}
    \|f(A)\| \leq 2 \sup _{z \in W(A) }|f(z)|.
\end{equation}
    \end{enumerate}
\end{Proposition}
\begin{proof}
To show $(a) $ is equivalent to  $(b)$ note that any $f$ which is analytic on an open neighbourhood of $W(A)$ will also lie in $\mathbf{A}(W(A))$. The result now follows from Mergelyan's theorem, which shows polynomials are dense in $\mathbf{A}(W(A))$. The implication $(b) \implies (c)$ is immediate. The implication $(c) \implies (a)$ follows from the following standard rescaling argument. For a polynomial $p$, let $$
p'(z) = \frac{p(z)}{\sup _{z \in W(A)}|p(z)|}.
$$
Then $p'$ satisfies the hypothesis of $(c)$, so $\| p'\| \leqslant 2$, i.e., $\|p(A)\| \leq 2 \sup _{z \in W(A)}|p(z)|$. The equivalence of $(a)$ and $(d)$ follows from an identical argument to the argument showing $(a)$ and $(b)$ are equivalent.
\end{proof}

The following propositions have also been observed previously, but as our later discussions rely on these results, we include their proofs here.
\begin{Proposition}\label{multi}
    Let $A$ satisfy Conjecture \ref{CC}, let $\lambda, \mu \in \mathbb{C}$ and denote the transpose of $A$ by $A^T$. Then
    \begin{enumerate}
        \item $B = \mu A+\lambda I_d$ satisfies the conjecture,
        \item $A^T$ satisfies the conjecture,
        \item $A^*$ (the adjoint of $A$) satisfies Conjecture \ref{CC},
    \end{enumerate}
\end{Proposition}

\begin{proof}
\noindent (a) For any polynomial $p(z)$, define $q(z)=p(\mu z+\lambda)$, which means $q(A)=p(\mu A+\lambda I_d)=p(B)$. Since Conjecture \ref{CC} holds for $A$, 
$
\|q(A)\| \leq 2 \sup _{z \in W(A)}|q(z)|
$ and so $\|p(B)\| \leq 2 \sup _{z \in W(A)}|p(\mu z+\lambda)|$. Since $W(B)=\mu W(A)+\lambda$, this implies $\quad\|q(B)\| \leq 2 \sup _{z \in W(B)}|p(z)|$.

\noindent (b) A short computation (see \cite[Chapter 1]{hornjohnson}) shows $W(A) = W(A^T)$, and as taking matrix transposes is norm invariant, for each polynomial $p$
$$
\|p(A^T)\| = \|p(A)^T\| = \|p(A)\| \leqslant 2 \sup_{z \in W(A)} |p(z)| = 2 \sup_{z \in W(A^T)} |p(z)| .
$$

\noindent (c) For a polynomial $p$, we write $p(z) = \sum_{k=0}^{N} p_k z^k$, where $p_k \in \mathbb{C}$. Note that $p(A^*) = (\Tilde{p}(A))^*$, where $\Tilde{p}(z)  = \sum_{k=0}^{N} \overline{p_k} z^k$. So as Conjecture \ref{CC} holds for the given matrix $A$,
\begin{align}
    \| p( A^*) \| &= \| (\Tilde{p}(A))^* \| = \| \Tilde{p}(A) \| \leqslant 2 \sup_{z \in W(A)} |   \Tilde{p} (z) | \\
    &\underbrace{=}_{*} 2 \sup_{z \in W(A^*)} |   \Tilde{p} (\overline{z}) | = 2 \sup_{z \in W(A^*)} |   \overline{\Tilde{p} (\overline{z})} | = 2 \sup_{z \in W(A^*)} |p (z) |,
\end{align}
where the starred equality holds because $z \in W(A) $ if and only if $\overline{z} \in W(A^*)$.
\end{proof}

For matrices $A_1, A_2, \ldots , A_n$ of dimensions $a_1 \times a_1, a_2 \times a_2, \ldots , a_n \times a_n$ respectively, let the  $(a_1 + a_2 + \cdots + a_n) \times (a_1 + a_2 + \cdots + a_n)$ matrix $A_1 \oplus A_2 \oplus \cdots \oplus A_n$ be defined by
\begin{equation}\label{reducible}
A_1 \oplus A_2 \oplus \cdots \oplus A_n = \begin{pmatrix}
A_1 & 0 & 0 & \cdots & 0 \\
0 & A_2 & 0 & \cdots & 0 \\
0 & 0 & A_3  & \cdots & 0 \\
\vdots & \vdots & \vdots & \ddots & \vdots \\
 0 &  0 & \cdots & \cdots & A_n
\end{pmatrix},
\end{equation}
where each $0$ denotes a block matrix of appropriate size. If a matrix, $M$, is unitarily equivalent to a matrix of the form \eqref{reducible} for $n >1$, then we say $M$ is \emph{reducible}. Otherwise we say $M$ is \emph{irreducible}.

\begin{Proposition}\label{directsum}
If matrices $A_1, A_2, \ldots , A_n$ satisfy Conjecture \ref{CC} then so does the matrix $A_1 \oplus A_2 \oplus \cdots \oplus A_n$.
\end{Proposition}

\begin{proof}
If $A_1, A_2, \ldots , A_n$ satisfy Conjecture \ref{CC} then for any polynomial $p$,
\begin{align}
    &\|p (A_1 \oplus A_2 \oplus \cdots \oplus A_n) \| \\
    &=\|p(A_1) \oplus p(A_2) \oplus \cdots \oplus p(A_n) \| \\
    &=\max \{ \|p(A_1)\|, \|p(A_2)\|, \ldots , \|p(A_n)\| \} \\
    &\leqslant 2 \max \left\{\sup_{z \in W(A_1)} | p(z)|, \sup_{z \in W(A_2)} | p(z)|, \ldots , \sup_{z \in W(A_n)} | p(z)|, \right\} \\
    & \leqslant 2 \sup_{z \in W(A_1 \oplus A_2 \oplus \cdots \oplus A_n)} | p(z)|,
\end{align}
where the final inequality holds because $$
\bigcup_{i=1}^n W(A_i)  \subseteq \text{conv} \{ W(A_1), W(A_2), \ldots , W(A_n) \} = W(A_1 \oplus A_2 \oplus \cdots \oplus A_n), 
$$
and so $A_1 \oplus A_2 \oplus \cdots \oplus A_n$ also satisfies the conjecture.
\end{proof}


Although the following proposition is in some sense new, it is fundamentally a result which follows quickly from the fact that Conjecture \ref{CC} holds for $2 \times 2$ matrices.

\begin{Proposition}\label{rank1}
Every rank one matrix satisfies Conjecture \ref{CC}.
\end{Proposition}

\begin{proof}
Every $n \times n$ rank one matrix, $A$, is of the form $ A(x) =  \langle x, v \rangle u$ for some $u,v \in \mathbb{C}^n$. Without loss of generality we may assume $u \notin \operatorname{span} v$, as if this is the case $A$ is a linear multiple of an orthogonal projection, which is normal and hence satisfies Conjecture \ref{CC}. Let $\mathcal{X} = \operatorname{span} \{u,v\}$. Then $A \mathcal{X} \subseteq \mathcal{X}$ and $A \mathcal{X}^{\perp} = \{ 0 \}$. Let $x_1, x_2$ be an orthonormal basis for $\mathcal{X}$ and let $x_3, x_4, ..., x_n$ be an orthonormal basis for $\mathcal{X}^{\perp}$, then with respect to the orthonormal basis $x_1, x_2, ..., x_n$, $A$ has the matrix representation
$$
A_1 \oplus Z = \begin{pmatrix}
A_1 & 0 \\
0 & Z 
\end{pmatrix},
$$
where $A_1$ is a $2 \times 2$ matrix and where $Z$ denotes the $(n-2) \times (n-2)$ matrix with each entry equal to $0$. By \cite[Theorem 1.1]{CC2matrix} every $2 \times 2$ matrix satisfies Conjecture \ref{CC}, and it is readily verifiable that $Z$ satisfies Conjecture \ref{CC}. So by Proposition \ref{directsum}, $A_1 \oplus Z$ satisfies Conjecture \ref{CC}, and thus by unitarily equivalence, so does $A$.
\end{proof}

Denote the $n \times n $ matrix with $1$ in the $ij'$th entry and $0$ in all other entries by $e_{ij}$. As $\lambda e_{ij}$ is rank one for each $\lambda \in \mathbb{C}$, by the proposition above $\lambda e_{ij}$ satisfies Conjecture \ref{CC}. This leads us to pose the following question.
\begin{Question}
    Is the set of all $n \times n $ matrices which satisfy Conjecture \ref{CC} closed under addition?
\end{Question}

\noindent Since the matrices $e_{ij}$ for $i,j =1, 2, \ldots, n$ form a basis for the space of all $n \times n $ matrices, a positive answer to the question above is equivalent to Conjecture \ref{CC} being true.

In the following, $\otimes$ denotes the tensor product (also known as Kronecker product) of matrices.
Further details of the construction of the tensor product of matrices may be found in \cite{symmetrictensors}.

\begin{Proposition}\label{newmulti}
Let $A$ satisfy Conjecture \ref{CC} and suppose that $N$ is a normal matrix. Then
$ N \otimes A$ and $ A \otimes N$ satisfy Conjecture \ref{CC}.
\end{Proposition}

\begin{proof}
First consider $ N \otimes A$. We have $U^* N U =  D $ for some unitary $U$ and diagonal $D$. It is known (see, for example, \cite[Section 3]{symmetrictensors}) that $
(U  \otimes I_d)^* = U^* \otimes I_d $. Thus $U  \otimes I_d$ is unitary and we have the unitary equivalence
\begin{align}
&(U  \otimes I_d)^* N \otimes A (U  \otimes I_d) = D  \otimes A .
\end{align}
Using the Kronecker product representation of $D  \otimes A$ we see $D  \otimes A$ is a direct sum of scalar multiples of $A$. Thus, by Propositions \ref{multi} and \ref{directsum}, $D  \otimes A$ will satisfy Conjecture \ref{CC} and by unitary equivalence, so will $N  \otimes A$.

To show $ A \otimes N$ satisfies Conjecture \ref{CC}, observe $N \otimes A$ and $A \otimes N$ are unitarily equivalent \cite[Proposition 2.3]{symmetrictensors}.
\end{proof}

The following result was stated in \cite[Page 3]{CCincreasing} without proof.
\begin{Proposition}\label{incr}
If Conjecture \ref{CC} holds for all $N \times N$ matrices then it holds for $n \times n$ matrices where $n < N$.
\end{Proposition}
\begin{proof}
It suffices to prove the statement in the case that $n+1 = N$, as then the result will follow inductively. Let $A$ be an $n \times n$ matrix and let $p$ be a polynomial. If 
\begin{equation}\label{WLOG}
    \| p(A) \| < \inf_{ z \in W(A)} |p(z)|,
\end{equation}
then $\| p(A) \| \leqslant 2 \sup_{ z \in W(A)} |p(z)|$. So we can assume without loss of generality that there exist a $d \in W(A)$ such that $|p(d)| \leqslant \| p(A) \|$. Set $B = \begin{pmatrix}
A & 0 \\
0 & d
\end{pmatrix}.
$ By assumption Conjecture \ref{CC} holds for $B$. So,
$W(B) = \text{conv} \{ W(A), d \} = W(A)$, and thus
$$
\| p(A) \| = \| p(B) \| \leqslant 2 \sup_{z \in W(B)} |p(z)| =  2 \sup_{z \in W(A)} |p(z)|.
$$ \qedhere
\end{proof}

The recent thesis \cite{CCthesis} contains many new results (some of which are also contained in a paper by the same author \cite{CCLi}) centred on norm attaining vectors of $\|f(A)\|$, where $f \in \mathbf{A}(W(A))$. By Proposition \ref{allequiv} we see that Conjecture \ref{CC} holds if and only if
\begin{equation}\label{CCsup}
    \sup_{f \in \mathbf{A}(W(A))} \frac{\|f(A)\|}{\max _{z \in W(A)}|f(z)|} \leqslant 2.
\end{equation}
It was first observed in \cite[Theorem 2.1]{CC2matrix} that there are functions $\hat{f}$ which attain the supremum in \eqref{CCsup} and that such functions are of the form $\mu B \circ \phi$, where $\mu \in \mathbb{C}$, $\phi$ is any conformal mapping from $W(A)^o$ to $\mathbb{D}$ and 
$$
B(z)=\exp (i \gamma) \prod_{j=1}^m \frac{z-\alpha_j}{1-\bar{\alpha}_j z}, \quad m \leq n-1, \quad\left|\alpha_j\right|<1, \quad \gamma \in \mathbb{R}
$$
is a Blaschke product of degree $m$. Such functions are called \emph{extremal functions} for $A$ and we will denote the extremal function for $A$ by $\hat{f}$ throughout. If one assumes that ${\max _{z \in W(A)}|\hat{f}(z)|} = 1$, then $\mu=1$, so that $\hat{f}$ is a function of the form $B \circ \phi$.

We summarise some results in \cite[Section 2.2]{CCthesis} related to the uniqueness of extremal functions in the following theorem.

\begin{Theorem}\label{Jmatrix}
    \begin{enumerate}
        \item For an $n \times n$ nilpotent Jordan block $$
         J=\begin{pmatrix}
0 & 1 & 0 & \cdots & 0 \\
0 & 0 & 1 & \cdots & 0 \\
0 & 0 & 0 & \ddots & 0 \\
0 & 0 & 0 & \cdots & 1 \\
0 & 0 & 0 & \cdots & 0
\end{pmatrix},
        $$ the unique (up to scalar multiplication) extremal function for $J$ is $\hat{f} (z) = z^{n-1}$.
        \item The extremal function for the matrix 
$$\begin{pmatrix}
    0 & 1 & 0 \\
    0 & 0 & \frac{1}{\sqrt{3}} \\
     0 & 0 & 0
\end{pmatrix}$$ is not unique.
    \end{enumerate}
\end{Theorem}

Through the use of extremal functions, an alternative proof that certain $2 \times 2$ matrices satisfy Conjecture \ref{CC} is presented in \cite[Section 2.3.1]{CCthesis}.



\section{Cyclicity}\label{3}

This section gives a new link between cyclicity and Conjecture \ref{CC}, and consequently we show that in order to prove Conjecture \ref{CC} in full generality one must only verify Conjecture \ref{CC} for the differentiation operator on a class of entire functions (see Theorem \ref{diffthm}). We also provide several shortened proofs of previously known results.

For an $n \times n$ matrix $A$ and a column vector $y \in \mathbb{C}^n$, the cyclic subspace generated by $y$ is $\operatorname{span} \{y, Ay, A^2y, \ldots \}$. We say that $y $ is cyclic for $A$ if the cyclic subspace generated by $y$ is $\mathbb{C}^n$. We say $A$ is cyclic if there exists a $y \in \mathbb{C}^n$ which is cyclic for $A$. Studying cyclicity properties of matrices (or the related field of Krylov subspaces) are active areas of research.

The following theorem establishes a novel link between Conjecture \ref{CC} and cyclicity. For an extremal function $\hat{f}$ for $A$, if $x$ is such that $\|\hat{f}(A) x\| = \|\hat{f}(A)\| \| x \|$, then $x$ is called a corresponding \emph{extremal vector} for $A$.

\begin{Theorem}\label{maincor}
    Let $A$ be a $n \times n $ matrix, let $\mathcal{V} \subseteq \mathbb{C}^n$ be such that $\dim \mathcal{V} \leqslant 2$ and let $A_{\mathcal{V}}:= P_{\mathcal{V}}A_{|\mathcal{V}} : \mathcal{V} \to \mathcal{V}$ (where $P_{\mathcal{V}}$ denotes the orthogonal projection onto ${\mathcal{V}}$) be the compression of $A$ to $\mathcal{V}$. 
    \begin{enumerate}
        \item If $ \| \hat{f} (A) \| = \| \hat{f} ( A_{\mathcal{V}})\|$, where $\hat{f}$ is an extremal function for $A$, then $A$ satisfies Conjecture \ref{CC}.
        \item If the cyclic subspace generated by an extremal vector $x$ for $A$ is one or two dimensional, then $A$ satisfies Conjecture \ref{CC}.
        \item If $A$ is a $3 \times 3$ matrix and an extremal vector $x \in \mathbb{C}^3$ is not cyclic for $A$, then $A$ satisfies Conjecture \ref{CC}. In particular $3 \times 3$ non-cyclic matrices satisfy Conjecture \ref{CC}.
    \end{enumerate}
\end{Theorem}
\begin{proof}
    $(a)$ As $A_{\mathcal{V}}$ is either a $1 \times 1$ or $2 \times 2$ matrix and both of these satisfy Conjecture \ref{CC} (see \cite{CC2matrix}), $A_{\mathcal{V}}$ satisfies Conjecture \ref{CC}, and it is readily checked that $W(A_{\mathcal{V}}) \subseteq W(A)$. So 
    $$
    \sup_{f \in \mathbf{A}(W(A))} \frac{\|f(A)\|}{\max _{z \in W(A)}|f(z)|} = \frac{\|\hat{f}(A)\|}{\max _{z \in W(A)}|\hat{f}(z)|} \leqslant \frac{\|\hat{f}(A_{\mathcal{V}})\|}{\max _{z \in W(A_{\mathcal{V}})}|\hat{f}(z)|} \leqslant   2.
    $$
   \noindent $(b)$ Set $\mathcal{V}$ to be the cyclic subspace generated by $x$. Since $A(\mathcal{V}) \subseteq \mathcal{V}$, we have $$\| \hat{f} (A) \| =  \| \hat{f} (A)_{|\mathcal{V}} \| =  \| \hat{f} ( A_{\mathcal{V}})\|.$$ Thus by part $(a)$, the matrix $A$ satisfies Conjecture \ref{CC}.

   \noindent $(c)$ If $A$ is a $3 \times 3$ matrix and $x$ is not cyclic for $A$, then clearly the cyclic subspace generated by $x$ is one or two dimensional. Thus the result follows from part $(b)$.
\end{proof}

\begin{Remark}
One could restate the theorem above so that if for each polynomial $p$, there exists a two dimensional subspace $\mathcal{V}$ (which can depend on $p$) such that $ \| p (A) \| = \| p ( A_{\mathcal{V}})\|$, then $A$ satisfies Conjecture \ref{CC}. This formulation will be more applicable when one does not know an extremal function for $A$.
\end{Remark}

Mimicking the arguments used in the theorem above shows that if Conjecture \ref{CC} holds for $n \times n$ matrices, then non-cyclic $(n+1) \times (n+1)$ matrices will also satisfy Conjecture \ref{CC}. We show this in the following corollary.

\begin{Corollary}\label{induct}
    \begin{enumerate}
        \item If Conjecture \ref{CC} holds for all $n \times n$ matrices, then Conjecture \ref{CC} holds for all $(n+1) \times (n+1)$ non-cyclic matrices.
        \item If Conjecture \ref{CC} holds for all cyclic matrices, then Conjecture \ref{CC} holds in full generality (i.e. for all matrices).
    \end{enumerate} 
\end{Corollary}

\begin{proof}
    (a) Assume Conjecture \ref{CC} holds for all $n \times n$ matrices, and let $A$ be a $(n+1) \times (n+1)$ non-cyclic matrix. Let $\hat{f}$ be extremal for $A$ and $x$ be a corresponding extremal vector. 

    Set $\mathcal{V} = \operatorname{span} \{x, Ax, A^2x, \ldots\}$ to be the cyclic subspace generated by $x$. Then $A(\mathcal{V}) \subseteq \mathcal{V}$. Since $\dim \mathcal{V} \leqslant n$, and we have assumed Conjecture \ref{CC} holds for $n \times n $ matrices, Proposition \ref{incr} shows that Conjecture \ref{CC} holds for $A_{|\mathcal{V}}$. Thus, as $\| \hat{f} (A) \| =  \| \hat{f} (A)_{|\mathcal{V}} \| =  \| \hat{f} ( A_{|\mathcal{V}})\|$ and $W(A_{|\mathcal{V}} )\subseteq W(A)$, it follows that
    $$
    \sup_{f \in \mathbf{A}(W(A))} \frac{\|f(A)\|}{\max _{z \in W(A)}|f(z)|} = \frac{\|\hat{f}(A)\|}{\max _{z \in W(A)}|\hat{f}(z)|} \leqslant \frac{\|\hat{f}(A_{|\mathcal{V}})\|}{\max _{z \in W(A_{|\mathcal{V}})}|\hat{f}(z)|} \leqslant   2.
    $$

    (b) Let Conjecture \ref{CC} hold for all cyclic matrices and let $P(n)$ be the statement ``Conjecture \ref{CC} holds for all $n \times n$ matrices.'' As previously mentioned $P(2)$ holds by \cite{CC2matrix}. If $P(n)$ holds then part (a) combined with the assumption that Conjecture \ref{CC} hold for all cyclic matrices shows $P(n+1)$ holds. Thus the result follows by induction. 
\end{proof}



\begin{Lemma}\label{equivcyclic}
    An $n \times n$ matrix $A$ is cyclic if and only if $A^*$ is cyclic.
\end{Lemma}

\begin{proof}
    This follows from the fact that a $A$ is cyclic if and only if the minimal and characteristic polynomials of $A$ are equal \cite[Lemma 25.6]{curtisbook}, and the fact that the minimal (respectively characteristic) polynomial of $A^*$ is the conjugate of the minimal (respectively characteristic) polynomial of $A$.
\end{proof}

For a cyclic $n \times n$ matrix $A$, let $v$ be cyclic for $A^*$ (such a $v$ is guaranteed to exist by Lemma \ref{equivcyclic}). For $u \in \mathbb{C}^n$, consider the entire function $\tilde{u}(z) := \langle e^{zA} u , v \rangle$. In the following theorem $E_{A,v} := \{\tilde{u}(z) : u \in \mathbb{C}^n \}$, where $\| \tilde{u} \|_{E_{A,v}} = \|u\|_{\mathbb{C}^n}$, and $\mathcal{D}_{A,v}: {E}_{A,v} \to {E}_{A,v}$ is the differentiation operator. The operator $\mathcal{D}_{A,v}$ is explicitly defined via $\mathcal{D}_{A,v} (\tilde{u}) (z) = \langle e^{zA} Au , v \rangle$. Theorem 1.1 in \cite{modelCC} states the following.

\begin{Theorem}\label{diffthm}
    If $A$ is an $n \times n$ cyclic matrix and $v$ is cyclic for $A^*$, then $A$ is unitarily equivalent to $\mathcal{D}_{A,v}$.
\end{Theorem}

The following theorem shows that proving the norm estimate \eqref{ECC} for the differentiation operator $\mathcal{D}_{A,v}$ associated with cyclic matrices would yield a full proof of Conjecture \ref{CC} for all matrices.

\begin{Theorem}
With the notation defined above, if for every cyclic matrix $A$ and every polynomial $p$, we have \begin{equation}\label{diffCC}
        \|p(\mathcal{D}_{A,v}) \|_{{E}_{A,v}} \leqslant 2 \sup _{z \in W(\mathcal{D}_{A,v})}|p(z)|
    \end{equation}
for some $v$ which is cyclic for $A^*$, then Conjecture \ref{CC} holds.
\end{Theorem}
\begin{proof}
    By Corollary \ref{induct} part (b), in order to prove Conjecture \ref{CC} in full generality, it is enough to prove Conjecture \ref{CC} for cyclic matrices. If $A$ is a cyclic matrix and $v$ is cyclic for $A^*$, Theorem 1.1 in \cite{modelCC} shows that $A$ is unitarily equivalent to $\mathcal{D}_{A,v}$, thus if \eqref{diffCC} holds, Conjecture \ref{CC} will hold for $A$. 
\end{proof}



In the article \cite{CCnilpotent}, Crouzeix gives a detailed analysis of $3 \times 3$ nilpotent matrices and ultimately proves that $3 \times 3$ nilpotent matrices satisfy Conjecture \ref{CC}. For a 2-nilpotent matrix, as the cyclic subspace generated by any vector is one or two dimensional, we immediately deduce the following corollary to part $(b)$ of Theorem \ref{maincor}.

\begin{Corollary}\label{nilpotent}
    Let $A$ be an $n \times n$ matrix such that $A^2 = 0$. Then the matrix $A$ satisfies Conjecture \ref{CC}.
\end{Corollary}
\begin{Remark}
Corollary \ref{nilpotent} provides a swift alternative algebraic proof of the result that has appeared several time previously in the literature. Recently, \cite[Theorem 6]{CCthesis}, proves Conjecture \ref{CC} holds for 2-nilpotent matrices. Alternatively, by combining \cite{minpolyCC} with \cite{CC2matrix} one can show that all matrices with minimum polynomial of degree $2$ satisfy Conjecture \ref{CC}. Crouzeix also noted in \cite{CCnilpotent} that for 2-nilpotent matrices, the numerical range is a disc, and thus must satisfy Conjecture \ref{CC}.
\end{Remark}

Theorem \ref{maincor} may lead one to consider if every $3 \times 3$ matrix has an extremal function with a corresponding extremal vector which is non-cyclic (if this were true, this would prove Conjecture \ref{CC} in the positive for $ 3 \times 3$ matrices). However, the following example shows this is not the case.

\begin{Example}
    Let $$
    J = \begin{pmatrix}
        0  & 1 & 0 \\
        0 & 0 & 1 \\
         0 & 0 & 0
    \end{pmatrix}
    $$ be a nilpotent Jordan block. Then by Theorem \ref{Jmatrix}, the (unique up to scalar multiplication) extremal function is $\hat{f} (z) =  z^2$. The (unique up to scalar multiplication) extremal vector for 
    $$
     J^2 = \begin{pmatrix}
        0 & 0 & 1 \\
        0 & 0 & 0 \\
        0 & 0 & 0
    \end{pmatrix}
    $$ is $\begin{pmatrix}
        0 \\
        0 \\
        1
    \end{pmatrix}$, which is cyclic for $J$. Thus the only extremal vectors for $J$ are cyclic.
\end{Example}

Nonetheless there are examples of cyclic matrices for which we can still apply Theorem~\ref{maincor} as the following example shows. 

\begin{Example}
    Let $$
    A = \begin{pmatrix}
        0  & 1 & 0 \\
        0 & 0 & 1-t \\
         0 & 0 & 0
    \end{pmatrix},
    $$ where $1 - \frac{1}{\sqrt{3}} \leqslant t \leqslant \sqrt{3} - 1$. It is readily checked that $\begin{pmatrix}
         0 \\
         0 \\
         1
    \end{pmatrix}$ is a cyclic vector for $A$. As highlighted in \cite[Section 2.2.2]{CCthesis} the extremal function for $A$ is $\hat{f}(z) = z$, and the corresponding extremal vector for $\hat{f}(A) = A$ is $\begin{pmatrix}
         0 \\
         1 \\
         0
    \end{pmatrix}$, which is not cyclic. Hence by part $(b)$ of Theorem \ref{maincor}, $A$ satisfies Conjecture \ref{CC}.
\end{Example}

Our workings lead us to ask the following questions.
\begin{Question}\label{Q1}
Which matrices $A$ have an extremal function $\hat{f}$ such that  $ \| \hat{f} (A) \| = \| \hat{f} ( A_{\mathcal{V}})\|$ for some two dimensional subspace $\mathcal{V}$?
\end{Question}

\begin{Question}\label{Q2}
Which matrices $A$ have the property that for each polynomial $p$  one can find a subspace $\mathcal{V}$ such that $\dim \mathcal{V} \leqslant 2$ and $ \| p (A) \| = \| p ( A_{\mathcal{V}})\|$?
\end{Question}

Notice that an affirmative answer to Questions \ref{Q1} or \ref{Q2} implies that the matrix $A$ satisfies Conjecture~\ref{CC}.

\section{Symmetric matrices and truncated Toeplitz operators}\label{4}

Complex symmetric operators are infinite dimensional generalisations of symmetric matrices. The past fifteen years has seen an explosion of research interest into complex symmetric operators. In pursuit of a proof of Conjecture \ref{CC} for symmetric matrices, in this section we investigate what role truncated Toeplitz operators play when studying the numerical ranges of symmetric matrices.

The Hardy space $H^2$ consists of all analytic functions on $\mathbb{D}$ whose Taylor coefficients are square summable, that is,
$$
H^2:= \left\{  f(z)  = \sum_{n \in \mathbb{N}_0} a_n z^n : \sum_{n \in \mathbb{N}_0 } |a_n|^2 < \infty \right\},
$$
which is a Hilbert space with the inner product defined by $$\left\langle \sum_{n \in \mathbb{N}_0} a_n z^n , \sum_{n \in \mathbb{N}_0} b_n z^n \right\rangle = \sum_{n \in \mathbb{N}_0} a_n \overline{b_n}.$$
It is well known that, for $f\in H^2$, the limits
$$
    \tilde f(e^{it}) := \lim_{r\to 1} f(re^{it})
$$
exist for almost every $t$, and $\tilde f\in L^2(\T)$ (here $\T = \partial \mathbb{D}$ denotes the unit circle). If we set $\widetilde H^2 := \{\tilde f : f\in H^2\} \subset L^2(\T)$, then $H^2$ and $\widetilde H^2$ are isometrically isomorphic, and $\widetilde H^2 = \{f \in L^2(\T) : f_n = 0\ \text{for}\ n<0\}$, where $f_n$ denotes the $n$th Fourier coefficient of $f$. We refer the reader to \cite{duren1970theory, nikolski2002operators} for a detailed background on the Hardy space.

We say a function $\theta \in H^2$ is \emph{inner} if $| \theta | = 1$ a.e. on $\mathbb{T}$. For an inner function $\theta$, we define the \emph{model space}, $K_\theta^2$, by $K_\theta^2 = (\theta H^2)^{\perp} \cap H^2 $. For example, if $\theta(z) = z^n$, then $K_\theta^2 = \text{span} \{ 1, z, ... ,z^{n-1} \}$. If $\theta(z) = \prod_{i=1}^n \frac{z - a_i}{1-\overline{a_i}z}$ for distinct $a_1, ..., a_n $ lying in the unit disc, then $K_\theta^2 = \text{span} \{ k_{a_1}, ..., k_{a_n} \}$ where $k_{a_i} = \frac{1}{1- \overline{a_i}z} \in H^2$ is the reproducing kernel at $a_i$. For further details on model spaces, see \cite{modelspacesbook}.

The \emph{truncated Toeplitz operator} (which we will abbreviate to TTO), $A_g^{\theta}: K_{\theta}^2 \to K_{\theta}^2$, having symbol $g \in L^{\infty} (\mathbb{T})$ is defined by
$$
A_g^{\theta} (f) = P_{\theta} (g f)
$$
where $P_\theta$ is the orthogonal projection $L^2 (\mathbb{T}) \to K_\theta^2$. In the special case when $\theta(z) = z^n$, $A_g^{\theta}$ is a $n \times n$ Toeplitz matrix. Further we note that compressed shift operators are TTOs with the symbol $g(z) = z$, and in particular, the authors of \cite{CCgorkin} recently emphasised their role in the study of Crouzeix’s conjecture.

TTOs have gained considerable interest from the operator theory community in the past fifteen years. We refer the reader to \cite{recentprogressGarcia, Isabellesurvey} for survey articles on the topic and to \cite{CCTTO, CCTTO2} for articles concerning the numerical ranges of TTOs.

Symmetric matrices are important for Crouzeix’s conjecture because they can mimic the numerical range of any matrix $A$, in the sense that for any matrix $A$ there exists a symmetric matrix, $S$, of the same dimensions as $A$ such that $W(A) = W(S)$ (see \cite{PSNR}). The following open question posed in \cite{recentprogressGarcia} shows direct sums of TTOs may play an important role in the understanding of numerical ranges of matrices and proving Crouzeix’s conjecture.

\begin{Question}\label{question}
Is every symmetric matrix unitarily equivalent to a direct sum of TTOs?
\end{Question}

Through conjugation maps (and specifically the use of orthonormal bases which are invariant under conjugation maps), one can show that every TTO is unitarily equivalent to a symmetric matrix (see \cite[Section 2]{symrep} for details). Conversely, Question \ref{question} is known to be true for $2$-by-$2$ and $3$-by-$3$ matrices (see \cite{recentprogressGarcia} and \cite{garciaUETTOanalyticsymbols} respectively), rank one matrices \cite{garciaUETTOanalyticsymbols} and to several inflations of truncated Toeplitz operators \cite{strousse}.

The following lemma uses these facts to show that TTOs serve as model operators for the numerical ranges of $3 \times 3$ matrices.

\begin{Lemma}
    For any $3 \times 3$ matrix $A$, there exists an inner function $\theta$ and $g_1, g_2 \in L^{\infty}(\mathbb{T})$ such that $W(A) = W(A_{g_1}^{\theta} \oplus A_{g_2}^{\theta})$ (summands may be 0).
\end{Lemma}
\begin{proof}
As previously mentioned, \cite{PSNR} shows that for any matrix $A$ there exists a symmetric matrix, $S$, of the same dimensions as $A$ such that $W(A) = W(S)$. If $S$ is irreducible, then $S$ is unitarily equivalent to a TTO \cite[Theorem 5.2]{garciaUETTOanalyticsymbols}. If $S$ is reducible, it is unitarily equivalent to a direct sum of a $2 \times 2$ and $1 \times 1$ matrix. Then since every $2 \times 2 $ matrix is unitarily equivalent to a TTO \cite[Theorem 5.2]{spatialiso}, $S$ is unitarily equivalent to a direct sum of two TTOs. 


Thus, in all cases for the matrix $S$ we can find (up to) two TTOs $A_{g_1}^{\theta}, A_{g_2}^{\theta}$ such that $W(S) = W(A_{g_1}^{\theta} \oplus A_{g_2}^{\theta})$.
\end{proof}

If Question \ref{question} was shown to be positive, one could make straightforward adaptations to the lemma above, and show that TTOs serve as model operators for the numerical ranges of all matrices. Related to this, we also have the following proposition.

\begin{Proposition}

The following statements are equivalent:
\begin{itemize}
    \item[(a)] Conjecture \ref{CC} holds for all TTOs on three dimensional model spaces with analytic symbols,
    \item[(b)] Conjecture \ref{CC} holds for all $ 3 \times 3$ symmetric matrices,
    \item[(c)] Conjecture \ref{CC} holds for all TTOs on three dimensional model spaces.
\end{itemize}

\end{Proposition}\label{symtto}
\begin{proof}
We first prove $a \implies b$. By \cite[Theorem 5.2]{garciaUETTOanalyticsymbols}, any $3 \times 3$ symmetric matrix $S$ is unitarily equivalent to (at least one of) the following: \begin{enumerate}
    \item[(i)] The direct sum of matrices.
    \item[(ii)] A rank one matrix.
    \item[(iii)] A TTO with an analytic symbol.
\end{enumerate}
As Conjecture \ref{CC} holds for $2 \times 2$ matrices (\cite[Theorem 1.1]{CC2matrix}), if (i) holds, by Proposition \ref{directsum}, $S$ will satisfy the conjecture. If (ii) holds, then Proposition \ref{rank1} shows $S$ will satisfy the conjecture. If (iii) holds, then clearly by assumption $S$ will satisfy the conjecture. Thus in all cases $S$ satisfies the conjecture.

The implication $b  \implies c$ follows from the previously mentioned fact that every TTO is unitarily equivalent to a symmetric matrix (see \cite[Section 2]{symrep} for details). The implication $c \implies a$ is immediate.
\end{proof}

\section*{Acknowledgements}
\noindent Ryan O'Loughlin is grateful to EPSRC for financial support.
\newline
\noindent Jani Virtanen was supported in part by Engineering and Physical Sciences Research Council (EPSRC) grant EP/X024555/1.

\section*{Declarations}
\noindent Declarations of interest: none.

\bibliographystyle{plain}
\bibliography{bibliogrpahy.bib}
\end{document}